\documentclass[11pt]{amsart}
\usepackage{amssymb}
\usepackage[all]{xy}
\usepackage{graphicx}
\usepackage{url}

\newtheorem{theorem}{Theorem}[section]
\newtheorem{proposition}[theorem]{Proposition}
\newtheorem{lemma}[theorem]{Lemma}
\newtheorem{corollary}[theorem]{Corollary}

\theoremstyle{definition}
\newtheorem{definition}[theorem]{Definition}
\newtheorem{example}[theorem]{Example}
\newtheorem{remark}[theorem]{Remark}

\makeatletter
\theoremstyle{theorem}%
\newtheorem*{rep@theorem}{\rep@title}
\newcommand{\newreptheorem}[2]{%
\newenvironment{rep#1}[1]{%
 \def\rep@title{#2 \ref{##1}}%
 \begin{rep@theorem}}%
 {\end{rep@theorem}}}
\makeatother

\newreptheorem{theorem}{Theorem}
\newreptheorem{corollary}{Corollary}

\numberwithin{equation}{section}
\numberwithin{figure}{section}
\numberwithin{table}{section}

\newcommand{\Z}{\mathbb{Z}}
\newcommand{\Q}{\mathbb{Q}}
\newcommand{\p}{\mathfrak{p}}
\renewcommand{\P}{\mathfrak{P}}
\newcommand{\oo}{\mathfrak{o}}
\newcommand{\B}{\mathcal{B}}
\newcommand{\Gal}{\operatorname{Gal}}
\newcommand{\Int}{\operatorname{Int}}
\newcommand{\id}{\operatorname{id}}
\newcommand{\pr}{\operatorname{pr}}
\newcommand{\hc}{\operatorname{hc}}
\newcommand{\op}{\operatorname{op}}
\newcommand{\sg}{\operatorname{sg}}
\newcommand{\lk}{\operatorname{lk}}
\newcommand{\Coker}{\operatorname{Coker}}

\newcommand{\incl}{\operatorname{incl}}
\newcommand{\tmod}{\mathrm{mod}\ }

\newcommand{\Leg}[2]{\left(\frac{#1}{#2}\right)}

\begin{document}

\title[Genus one homologically fibered knots]
{Every lens space contains a genus one homologically fibered knot}
\author[Y. Nozaki]{Yuta Nozaki}
\address{Graduate School of Mathematical Sciences, the University of Tokyo \\
3-8-1 Komaba, Meguro-ku, Tokyo, 153-8914 \\
Japan}
\email{nozaki@ms.u-tokyo.ac.jp}

\maketitle

\begin{abstract}
 We prove that every lens space contains a genus one homologically fibered knot, which is contrast to the fact that some lens spaces contain no genus one fibered knot.
 In the proof, the Chebotarev density theorem and binary quadratic forms in number theory play a key role.
 We also discuss the Alexander polynomial of homologically fibered knots.
\end{abstract}

\setcounter{tocdepth}{1}
\tableofcontents

\section{Introduction}\label{sec:Intro}

It is well known that every connected oriented closed (namely, compact and without boundary) 3-manifold $X$ contains a fibered knot.
In other words, $X$ admits an open book decomposition with connected binding.
The minimal genus $\op(X)$ of pages of all such open book decompositions of $X$ is a fundamental invariant of $X$.
For instance, $\op(X)=0$ if and only if $X \cong S^3$.
The concept of the invariant $\op(X)$ is similar to the support genus $\sg(\xi)$ introduced by Etnyre and Ozbagci~\cite{EtOz08}, where $\xi$ is a contact structure on $X$.
The invariant $\sg(\xi)$ is defined to be the minimal genus of a page of all open book decompositions (whose bindings are not necessarily connected) of $X$ supporting $\xi$.

Morimoto~\cite{Mor89} started to study genus one fibered knots (GOF-knots) in lens spaces, and Baker~\cite[Theorem~4.3]{Bak14b} completely determined which lens space $L(p,q)$ contains a GOF-knot, that is, we already know when $\op(L(p,q))=1$ holds.
However, computation of $\op(X)$ is difficult in general.

Sakasai~\cite[Remark 6.10]{Sak17} introduced a homological analogue $\hc(X)$ of $\op(X)$, which is roughly defined to be the minimal genus of surfaces $\Sigma$ whose complements are homologically $\Sigma\times[-1,1]$.
Precisely, $\hc(X)$ is defined in terms of homology cobordisms or homologically fibered knots (see Definitions~\ref{def:hc}), and $\hc(X) \leq \op(X)$ holds by definition.
The author was informed by Sakasai the following sufficient condition for $\hc(L(p,q))=1$ when $q$ is odd: $p(p+4)$ or $p(p-4)$ is a quadratic residue $\tmod q$.

The purpose of this paper is to prove the following theorem and corollary which contain new results on the computation of $\hc(X)$ for various 3-manifolds $X$.

\begin{theorem}\label{thm:hc=1}
 $\hc(L(p,q))=1$ holds for any lens space $L(p,q)$, or equivalently $L(p,q)$ contains a genus one homologically fibered knot.
\end{theorem}

Let $d(G)$ denote the minimum number of generators of a group $G$.

\begin{corollary}\label{cor:List_of_hc}
 The following hold for $g \in \Z_{\geq1}$ and $p, n \in \Z_{\geq2}$.
\begin{enumerate}
 \item $\hc(X)=0$ if and only if $H_1(X) = 0$.
 \item $\hc(X)=g$ if $H_1(X)$ is isomorphic to $\Z^{2g-1}$ or $\Z^{2g}$.
 \item $\hc(X)=1$ if $H_1(X) \cong \Z/p$.
 \item If $X$ is a rational homology $3$-sphere and the subgroup of $H_1(X)$ consisting of $2$-torsions is cyclic \textup{(}possibly trivial\textup{)}, then $\hc(X) \leq d(H_1(X))$.
 \item $\hc(L(p_1,q_1) \sharp L(p_2,q_2))=2$ if $p_1$ divides $p_2$ and neither $q_1q_2$ nor $-q_1q_2$ is a quadratic residue $\tmod p_1$.
 \item Suppose $H_1(X) \cong \Z\oplus\Z/p$.
 Then, $\hc(X)=1$ if there is $q \in \Z$ such that the torsion linking form $\lambda_X$ is isomorphic to $(q/p)$ and $q$ or $-q$ is a quadratic residue $\tmod p$.
 Otherwise, $\hc(X)=2$.
\end{enumerate}
\end{corollary}

Note that $\hc(X)$ is \emph{not} determined only by the isomorphism class of $H_1(X)$.
Indeed, for two 3-manifolds $X_i = L(5,1) \sharp L(5,i)$ ($i=1,2$), we have $\hc(X_1)=1 \neq \hc(X_2)=2$ (\cite[Remark~6.10]{Sak17}).
On the other hand, Sakasai proved the following theorem.

\begin{theorem}[{\cite[Remark~6.10]{Sak17}}]\label{thm:Sakasai}
 The invariant $\hc(X)$ depends only on the isomorphism class of the pair of $H_1(X)$ and the torsion linking form $\lambda_X \colon TH_1(X)\times TH_1(X) \to \Q/\Z$, where $TH_1(X)$ denotes the torsion subgroup of $H_1(X):=H_1(X;\Z)$.
\end{theorem}

In fact, the torsion linking form of $X_i$ is $(1/5)\oplus(i/5)$, and they are not isomorphic for $i=1,2$.
In order to prove Theorem~\ref{thm:hc=1}, we find a surface $\Sigma \subset L(p,q)$ of genus one whose complement $L(p,q)\setminus\Int(\Sigma\times[-1,1])$ is a homology cobordism.
(It is easy to see that $\hc(X)=0$ if and only if $X$ is an integral homology 3-sphere.)
The following theorem (to be proved in Section~\ref{sec:PrimeDivMod} by using the Chebotarev density theorem) and a well-known fact about binary quadratic forms allow us to construct a desired surface $\Sigma$.

\begin{theorem}\label{thm:PrimeDivMod}
 Let $m \in \Z$ and $n \in \Z_{>0}\setminus\{5\}$ be coprime.
 Then there exist $\varepsilon \in \{1,-1\}$ and an odd prime $l$ such that the congruence equation $nx(x+1) \equiv \varepsilon \mod l$ is solvable and $l \equiv m \mod n$.
\end{theorem}

In Section~\ref{sec:HomologyCobordism}, we shall review the invariant $\hc(X)$ and prove Theorem~\ref{thm:hc=1}.
Section~\ref{sec:PrimeDivMod} is devoted to proving Theorem~\ref{thm:PrimeDivMod} based on number theory.
In the final section, we focus on Seifert matrices of (homologically fibered) knots which are useful to study $\hc(X)$.
Throughout this paper, $\Sigma_{g,b}$ denotes a connected oriented compact surface of genus $g$ with $b$ boundary components, and $L(p,q)$ denotes the lens space obtained from $S^3$ by Dehn surgery on an unknot along the slope $-p/q$, where $p \geq 2$ and $q$ are coprime.

\subsection*{Acknowledgments}
The author would like to thank Takuya Sakasai and Gw\'ena\"el Massuyeau for their various discussion.
Also, he would like to express his gratitude to Mutsuro Somekawa and Ippei Nagamachi for their useful comments to prove Theorem~\ref{thm:PrimeDivMod}.
The author wishes to express his thanks to Jun Ueki and the referees for their careful reading of the manuscript and for their various comments.
He wishes to be grateful to Institut de Recherche Math\'ematique Avanc\'ee, Universit\'e de Strasbourg, where most of this paper was written, for the hospitality.
Finally, this work was supported by the Program for Leading Graduate Schools, MEXT, Japan and JSPS KAKENHI Grant Number 16J07859.

\section{Homology cobordisms and proof of Theorem~\ref{thm:hc=1}}\label{sec:HomologyCobordism}
We first review homology cobordisms following Garoufalidis and Levine~\cite[Section~2.4]{GaLe05}.

\begin{definition}
 A \emph{homology cobordism over $\Sigma_{g,1}$} is a triad $(M,i_+,i_-)$, where $M$ is an oriented compact 3-manifold and $i_+, i_-\colon \Sigma_{g,1} \to \partial M$ are embeddings satisfying
\begin{itemize}
 \item $i_+$ is orientation-preserving and $i_-$ is orientation-reversing;
 \item $i_+|_{\partial\Sigma_{g,1}} = i_-|_{\partial\Sigma_{g,1}}$;
 \item $i_+(\Sigma_{g,1}) \cup i_-(\Sigma_{g,1}) = \partial M$ and $i_+(\Sigma_{g,1}) \cap i_-(\Sigma_{g,1}) = i_\pm(\partial\Sigma_{g,1})$;
 \item The induced maps $(i_+)_\ast, (i_-)_\ast\colon H_\ast(\Sigma_{g,1}) \to H_\ast(M)$ are isomorphisms.
\end{itemize}
\end{definition}

Note that the fourth condition is equivalent to the condition that $M$ is connected and $i_\pm$ induce isomorphisms on $H_1(-)$.
Sakasai~\cite[Definition~6.9, Remark~6.10]{Sak17} introduced the following invariant of 3-manifolds by using homology cobordisms.

\begin{definition}\label{def:hc}
 For a connected oriented closed 3-manifold $X$, $\hc(X) \in \Z_{\geq0}$ is defined by
 \[\hc(X) := \min\{g\in\Z_{\geq0} \mid \text{$X\cong C_M$ for some $(M,i_+,i_-)$}\},\]
 where $C_M$ is the \emph{closure} of a homology cobordism $(M,i_+,i_-)$ over $\Sigma_{g,1}$ defined by
 \[C_M := M/(i_+(x)\sim i_-(x),\ x\in\Sigma_{g,1}).\]
\end{definition}

\begin{remark}
 The inequalities $\hc(X) \leq \op(X)$ and $\hc(X\sharp Y) \leq \hc(X)+\hc(Y)$ hold for any $X, Y$ by definition.
 The gap $\op(X)-\hc(X)$ can be arbitrarily large.
 Indeed, let $X$ be the connected sum of $n$ copies of the Poincar\'e homology 3-sphere.
 Then we conclude that $\hc(X)=0$ and $\op(X) \geq n$ by Remark~\ref{rem:inequality}.
\end{remark}

For an embedding $\iota\colon \Sigma_{g,1} \xrightarrow{\cong} \Sigma \subset X$, we obtain the triad $(X\setminus \Int(\Sigma\times[-1,1]),\iota^-,\iota^+)$, where $\iota^\pm := \iota\times(\pm1)$.
Whether this triad is a homology cobordism or not depends only on the image of $\iota$, and we simply say that $X\setminus \Int(\Sigma\times[-1,1])$ is a homology cobordism if the triad is so.

\begin{figure}[h]
 \centering
 \includegraphics[height=10em]{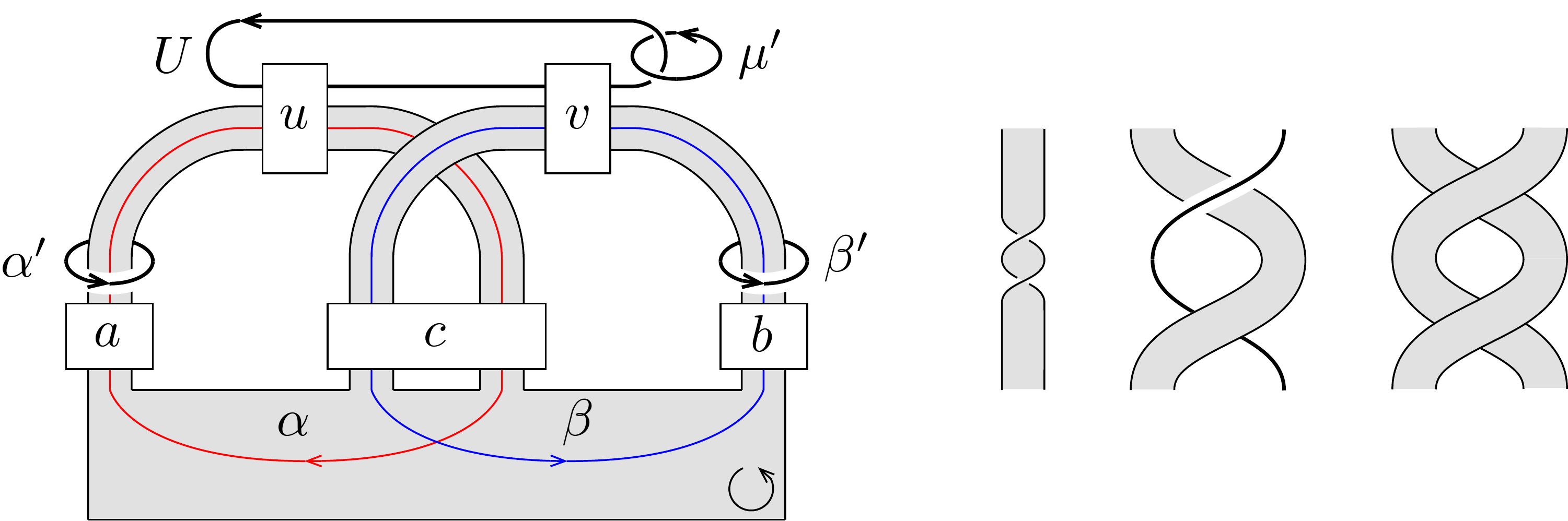}
 \caption{The surface $\Sigma_{a,b,c,u,v} \subset S^3_U(-p/q) = L(p,q)$, where $a,b,c,u,v \in \Z$.
 The box with $n\geq0$ (resp.\ $n<0$) represents suitable $|n|$ positive (resp.\ negative) full twists on the right.}
 \label{fig:Sabcuv}
\end{figure}

The following key lemma is a corollary of Proposition~\ref{prop:SeifertMatrix}, though we give a direct proof in this section.
Note that for coprime integers $p,q$ there is an integer solution $(x,y)=(r_0,s_0)$ of $py-qx=1$, and then the other solutions have the form $(r_k,s_k):=(r_0+kp,s_0+kq)$ for each $k \in \Z$.

\begin{lemma}\label{lem:Sabcuv}
 The complement $L(p,q) \setminus \Int(\Sigma\times[-1,1])$ of the surface $\Sigma=\Sigma_{a,b,c,u,v}$ illustrated in Figure~\textup{\ref{fig:Sabcuv}} is a homology cobordism if and only if there exist $\varepsilon \in \{1,-1\}$ and $k \in \Z$ such that
\begin{align}\label{eq:Sabcuv}
 \begin{pmatrix}
 bu^2+(2c+1)uv+av^2 \\
 c^2+c-ab
\end{pmatrix}
 = \varepsilon
\begin{pmatrix}
 r_k \\
 s_k
\end{pmatrix}.
\end{align}
\end{lemma}

\begin{proof}
 Let $N(U)$ denote a (closed) tubular neighborhood of $U$ disjoint from $\Sigma\times[-1,1]$ and $\mu'$.
 Consider the Mayer-Vietoris sequence for 
 \[L(p,q) \setminus \Int(\Sigma\times[-1,1]) = (S^3 \setminus \Int(\Sigma\times[-1,1] \sqcup N(U))) \cup_{\partial} D^2\times S^1\]
 where the solid torus $D^2\times S^1$ is glued by (the isotopy class of) a homeomorphism $\partial(D^2\times S^1) \to \partial N(U)$ corresponding to $
\begin{pmatrix}
 -p & r_0 \\
 q & -s_0
\end{pmatrix}
 \in SL(2,\Z)$.
 Then we have $H_1(L(p,q) \setminus \Sigma) \cong \Coker \iota$, where $\iota$ is the map
 \[\iota\colon H_1(\partial(D^2\times S^1)) \rightarrowtail H_1(S^3 \setminus (\Sigma \sqcup N(U))) \oplus H_1(D^2\times S^1).\]
 By the definition of Dehn surgery, the matrix of $\iota$ with respect to the bases $\{\mu,\lambda\}$ and $\{\alpha',\beta',\mu',\lambda\}$ is
 \[A:=
\begin{pmatrix}
 qu & -s_0u \\
 -qv & s_0v \\
 -p & r_0 \\
 0 & 1
\end{pmatrix},\]
 where $\mu$ and $\lambda$ denote a meridian and longitude on $\partial(D^2\times S^1)$ respectively.
 
 Suppose that the complement is a homology cobordism, namely $\{\alpha_\sigma,\beta_\sigma\}$ is a basis of $\Coker \iota$ for each $\sigma=\pm$, where $\alpha_\pm:=\alpha\times\{\pm1\}$, $\beta_\pm:=\beta\times\{\pm1\}$.
 Then there is a basis $\B_\sigma=\{\alpha_\sigma,\beta_\sigma,\ast\}$ of $H_1(S^3 \setminus (\Sigma \sqcup N(U)))$ for each $\sigma=\pm$.
 Let $Q_\sigma \in GL(3,\Z)$ be the matrix changing the basis $\{\alpha',\beta',\mu'\}$ to $\B_\sigma$.
 Then we see that
 \[Q_+ = 
\begin{pmatrix}
 a & c & \ast \\
 c+1 & b & \ast \\
 u & -v & \ast
\end{pmatrix},\quad
 Q_- = 
\begin{pmatrix}
 a & c+1 & \ast \\
 c & b & \ast \\
 u & -v & \ast
\end{pmatrix}.\]
 Since $\{\alpha_\sigma,\beta_\sigma\}$ is a basis of $\Coker\iota$, the $(2\times2)$-matrix at the bottom of the new matrix $(Q_\sigma^{-1} \oplus I_1)A$ of $\iota$ must belong to $GL(2,\Z)$.
 Hence the absolute value of its $(1,3)$-entry equals 1, namely one has
 \[(Q_\sigma)^{-1}
\begin{pmatrix}
 qu \\ -qv \\ -p
\end{pmatrix}
=
\begin{pmatrix}
 \ast \\ \ast \\ 1
\end{pmatrix}.\]
 Here, since choosing another basis $\B_\pm'=\{\alpha_\sigma,\beta_\sigma,\ast'\}$ corresponds to elementary row operations using the entry 1, we may assume $\ast$'s in the last equality are zero.
 Then we conclude that
 \[Q_+ = 
\begin{pmatrix}
 a & c & qu \\
 c+1 & b & -qv \\
 u & -v & -p
\end{pmatrix},\quad
 Q_- = 
\begin{pmatrix}
 a & c+1 & qu \\
 c & b & -qv \\
 u & -v & -p
\end{pmatrix}.\]
 It follows from $|{\det Q'_\pm}|=1$ that
 \[|p(c^2+c-ab)-q(bu^2+(2c+1)uv+av^2)| = 1.\]
 This completes one direction, and the other is shown by reversing the above argument.
\end{proof}

Let us prove the main theorem by using Lemma~\ref{lem:Sabcuv}, Theorem~\ref{thm:PrimeDivMod} (to be proved later) and the following fact (see, for example, \cite[Section~5.3]{Bak84}):
For $n, \Delta \in \Z$, the congruence equation $z^2 \equiv \Delta \mod 4n$ is solvable if and only if there is a binary quadratic form $f(x,y)$ with discriminant $\Delta$ such that $f(x,y)=n$ has a primitive solution.

\begin{proof}[Proof of Theorem~\textup{\ref{thm:hc=1}}]
 We first consider the case $p=5$.
 Since $L(5,1) \cong L(5,4)$ and $L(5,2) \cong L(5,3)$, it is enough to see the cases $q=1,3$.
 Lemma~\ref{lem:Sabcuv} shows that $L(5,1)\setminus\Sigma_{0,0,0,1,1}$ is a homology cobordism.
 Indeed, $(r_0,s_0)=(-1,0)$ and $\varepsilon=-1$ satisfy equation \eqref{eq:Sabcuv}.
 Similarly, letting $(r_0,s_0)=(3,2)$ and $\varepsilon=1$, $L(5,3)\setminus\Sigma_{0,0,1,1,1}$ is a homology cobordism.
 
 Suppose $p \neq 5$.
 Putting $m=r_0$, $n=p$ in Theorem~\ref{thm:PrimeDivMod}, we conclude that there are $\varepsilon \in \{1,-1\}$ and $k \in \Z$ such that $px(x+1) \equiv \varepsilon \mod r_k$ has a solution $x=x_0$.
 Here, $z_0:=1+2x_0$ satisfies $z_0^2 \equiv 1+4\varepsilon s_k \mod 4\varepsilon r_k$.
 By the above fact, there is a quadratic form $f(x,y)=a'x^2+b'xy+c'y^2$ such that $b'^2-4a'c' = 1+4\varepsilon s_k$ and $f(x,y)=\varepsilon r_k$ has a solution $(x,y)=(u,v)$.
 Then, integers $a:=c'$, $b:=a'$ and $c:=(b'-1)/2$ satisfy 
 \[\begin{pmatrix}
 bu^2+(2c+1)uv+av^2 \\
 c^2+c-ab
\end{pmatrix}
=\varepsilon
\begin{pmatrix}
 r_k \\
 s_k
\end{pmatrix}.\]
 Therefore, we conclude from Lemma~\ref{lem:Sabcuv} that $\hc(L(p,q))=1$.
\end{proof}

\begin{remark}
 It follows from the proof of the above fact about quadratic forms that integers $a,b,c,u,v$ are represented by $z_0,\varepsilon,r_k,s_k$.
 However, it seems difficult to represent these integers explicitly by $p,q$.
\end{remark}

We sum up values or estimates of $\hc(X)$ for various $X$'s as a corollary of Theorem~\ref{thm:hc=1}.
Recall that $d(G)$ denotes the minimum number of generators of a group $G$.
(We set $d(\{1\})=0$ by convention.)

\begin{repcorollary}{cor:List_of_hc}
 The following hold for $g \in \Z_{\geq1}$ and $p, n \in \Z_{\geq2}$.
\begin{enumerate}
 \item $\hc(X)=0$ if and only if $H_1(X) = 0$.
 \item $\hc(X)=g$ if $H_1(X)$ is isomorphic to $\Z^{2g-1}$ or $\Z^{2g}$.
 \item $\hc(X)=1$ if $H_1(X) \cong \Z/p$.
 \item If $X$ is a rational homology $3$-sphere and the subgroup of $H_1(X)$ consisting of $2$-torsions is cyclic \textup{(}possibly trivial\textup{)}, then $\hc(X) \leq d(H_1(X))$.
 \item $\hc(L(p_1,q_1) \sharp L(p_2,q_2))=2$ if $p_1$ divides $p_2$ and neither $q_1q_2$ nor $-q_1q_2$ is a quadratic residue $\tmod p_1$. 
 \item Suppose $H_1(X) \cong \Z\oplus\Z/p$.
 Then, $\hc(X)=1$ if there is $q \in \Z$ such that the torsion linking form $\lambda_X$ is isomorphic to $(q/p)$ and $q$ or $-q$ is a quadratic residue $\tmod p$.
 Otherwise, $\hc(X)=2$.
\end{enumerate}
\end{repcorollary}

\begin{remark}\label{rem:inequality}
 We have well-known inequalities related to Corollary~\ref{cor:List_of_hc}~(4):
\begin{itemize}
 \item $d(H_1(X)) \leq 2\hc(X) \leq 2\op(X)$ and
 \item $d(H_1(X)) \leq d(\pi_1(X)) \leq g(X) \leq 2\op(X)$,
\end{itemize}
 where $g(X)$ denotes the Heegaard genus of $X$.
 The first inequality is found in \cite[Remark~6.1]{Sak17}.
 The second inequality in the second row follows that a Heegaard splitting of genus $g$ gives a presentation of $\pi_1(X)$ with $g$ generators.
 Also, when $X$ admits an open book decomposition with a page $\Sigma_{g,1}$, the union of two pages divides $X$ into two handlebodies of genus $2g$, and thus the third inequality holds.
 
\end{remark}

\begin{proof}[Proof of Corollary~\textup{\ref{cor:List_of_hc}}~\textup{(1)--(4)}]
 We first prove (1).
 If $H_1(X)=0$, then the complement of any disk in $X$ is a homology cobordism over $\Sigma_{0,1}$.
 The converse follows from the inequality $d(H_1(X)) \leq 2\hc(X)$ in Remark~\ref{rem:inequality}.
 (2) is due to Sakasai~\cite{Sak17}.
 (3) is a direct consequence of Theorems~\ref{thm:hc=1} and \ref{thm:Sakasai} since a non-degenerate symmetric bilinear form on $\Z/p$ is isomorphic to $(q/p) = \lambda_{L(p,q)}$ for some $q$.
 
 We next prove (4).
 Since $H_1(X)$ is finite, it is isomorphic to $\bigoplus_{i=1}^s \Z/p_i$ for some $p_i$'s with $p_i \mid p_{i+1}$, where $s:=d(H_1(X))$.
 It follows from \cite[Theorem~(4)]{Wal63} that $\lambda_X$ is isomorphic to $\bigoplus_{i=1}^s (q_i/p_i)$ for some $q_i$'s, hence $\lambda_X$ is isomorphic to $\lambda_{L(p_1,q_1) \sharp\cdots\sharp L(p_s,q_s)}$.
 Therefore, Theorem~\ref{thm:hc=1} shows $\hc(X) \leq s$.
\end{proof}

The proofs of (5) and (6) are given in the end of Section~\ref{sec:HFK} since we need results of Section~\ref{sec:HFK}.

\section{Number theory and proof of Theorem~\ref{thm:PrimeDivMod}}\label{sec:PrimeDivMod}
The goal of this section is to prove Theorem~\ref{thm:PrimeDivMod}, which was used in the proof of Theorem~\ref{thm:hc=1}.
We briefly review the Artin symbol of a prime ideal only for the case of abelian extensions following \cite[Chapter~X, Section~1]{Lan94} and \cite[Chapter~VI, Section~7]{Neu99}.
Let $k$ be a number field (assumed to be finite over $\Q$) and $K/k$ a finite abelian extension.
Let $\p\neq0$ be a prime ideal of $\oo_k$ unramified in $K/k$, where $\oo_k$ denotes the ring of integers of $k$.
Let $\P$ be a prime ideal of $\oo_K$ lying above $\p$, that is, $\P \cap k = \p$.
Then there exists a unique element of the decomposition group $\{\sigma \in \Gal(K/k) \mid \sigma(\P)=\P\}$ of $\P$ satisfying $\sigma(x) \equiv x^{N(\p)} \mod \P$ for all $x \in \oo_K$, where $N(\p)$ denotes the order of the residue field $\oo_k/\p$.
The element is independent of the choice of $\P$.
It is denoted by $\Leg{K/k}{\p}$ and called the \emph{Artin symbol} of $\p$.

\begin{example}[{\cite[Chapter~X, Section~1]{Lan94}}]\label{ex:ArtinSymbol}
 We review two well-known examples used in this paper.
 Let $a \in \Z$ be not a square, $\Delta_a$ be the discriminant of $\Q(\sqrt{a})$, $p$ be a prime with $\gcd(\Delta_a,p)=1$.
 Then, $p\Z$ is unramified and we deduce
 \[\Leg{\Q(\sqrt{a})/\Q}{p\Z} = 
\begin{cases}
 \id_{\Q(\sqrt{a})} & \text{if $\Delta_a$ is a quadratic residue $\tmod p$,} \\
 [\sqrt{a} \mapsto -\sqrt{a}\,] &  \text{otherwise.}
\end{cases}\]
 
 Next, for a primitive $n$th root of unity $\zeta_n$ and a prime $p$ with $\gcd(n,p)=1$, the ideal $p\Z$ is unramified and $\Leg{\Q(\zeta_n)/\Q}{p\Z} = [\zeta_n \mapsto \zeta_n^p]$ holds.
\end{example}

\begin{lemma}\label{lem:Chebotarev}
 Let $K_1/k$, $K_2/k$ be abelian extensions, $C$ be a subset of $G_1\times G_2$ with $\iota^{-1}(C) \neq \emptyset$, where $G_i$ denotes $\Gal(K_i/k)$ and $\iota\colon \Gal(K_1K_2/k) \to G_1\times G_2$ is defined by $\iota(\sigma)=(\sigma|_{K_1},\sigma|_{K_2})$.
 Then
 \[S:=\left\{ \p \biggm| \text{\parbox{11em}{$\p$ is a prime ideal of $\oo_k$\\ unramified in $K_1$ and $K_2$} and $\left( \Leg{K_1/k}{\p}, \Leg{K_2/k}{\p} \right) \in C$} \right\}\] 
 is an infinite set.
\end{lemma}

\begin{proof}
 First note that $K_1K_2/k$ is a Galois extension with $G=\Gal(K_1K_2/k)$ abelian, and the homomorphism $\iota\colon G \to G_1\times G_2$ is injective.
 We define the set $S'$ by
 \[S':= \left\{ \p \biggm| \text{\parbox{10em}{$\p$ is a prime ideal of $\oo_k$\\ unramified in $K_1K_2$} and $\Leg{K_1K_2/k}{\p} \in \iota^{-1}(C)$} \right\}.\]
 By the Chebotarev density theorem (see, for example, \cite[Chapter~VIII, Theorem~10]{Lan94}, \cite[Chapter~VII, Theorem~13.4]{Neu99}), we have $d(S') = |\iota^{-1}(C)|/|G| > 0$, where $d(S')$ is the Dirichlet density of $S'$ (see \cite[Chapter~VIII, Section~4]{Lan94}, \cite[Chapter~VII, Section~13]{Neu99}).
 Hence $S'$ is infinite.
 On the other hand, the consistency property \cite[Chapter~X, Section~1]{Lan94} asserts $S' \subset S$, and thus $S$ is also infinite.
\end{proof}

\begin{lemma}\label{lem:QuadResMod}
 For $a,m \in \Z$, $n \in \Z_{>0}$ with $\sqrt{a} \notin \Q(\zeta_n)$ and $\gcd(m,n)=1$,
 \[S := \{p \mid \text{$p$ is a prime, $a$ is a quadratic residue $\tmod p$ and $p \equiv m \mod n$}\}\]
 is an infinite set.
\end{lemma}

\begin{proof}
 Put $k=\Q$, $K_1=k(\sqrt{a})$, $K_2=k(\zeta_n)$ and $C=\{(\id_{K_1},[\zeta_n \mapsto \zeta_n^m])\}$.
 Since $K_1 \cap K_2 =k$, the map $\iota$ in Lemma~\ref{lem:Chebotarev} is an isomorphism, and thus $\iota^{-1}(C) \neq \emptyset$.
 It follows from Lemma~\ref{lem:Chebotarev} that 
 \[ \left\{ p \biggm| \text{\parbox{7.5em}{$p>\max\{\Delta_a,n\}$\\ is a prime}, $\Leg{K_1/k}{p\Z}=\id_{K_1}$ and $\Leg{K_2/k}{p\Z}=[\zeta_n \mapsto \zeta_n^m]$} \right\} \]
 is infinite.
 Here, Example~\ref{ex:ArtinSymbol} implies that this set is contained in $S$.
\end{proof}

\begin{lemma}\label{lem:Weintraub}
 For a positive integer $n \neq 5$, $\sqrt{n(n+4)}$ or $\sqrt{n(n-4)}$ does not belong to the $n$th cyclotomic field $\Q(\zeta_n)$.
\end{lemma}

\begin{proof}
 The cases $n=1,2$ are obvious since $\Q(\zeta_n) = \Q$.
 For $n>2$, there are four cases: (i) $v_2(n)=0$, (ii) $v_2(n)=1$, (iii) $v_2(n)=2$, (iv) $v_2(n)\geq3$, where $v_2$ is the 2-adic valuation for $\Z$.
 We only discuss (i) and (iv), and the other cases are shown similarly.
 In general, $\sqrt{a} \in \Q(\zeta_n)$ if and only if $a \in \Z$ is a product of a square and some integers in
 $$\{(-1)^{(p-1)/2}p \mid \text{$p$ is an odd prime factor of $n$}\} \cup 
\begin{cases}
 \emptyset & \text{if $v_2(n)<2$,} \\
 \{-1\} & \text{if $v_2(n)=2$,} \\
 \{-1,2\} & \text{if $v_2(n)>2$}
\end{cases}
 $$
 (see, for example, \cite[Corollary~4.5.4]{Wei09}).
 
 (i) Assume that $\sqrt{n(n+4)}, \sqrt{n(n-4)} \in \Q(\zeta_n)$.
 Then $n(n\pm4)$ must be certain products as mentioned above.
 It follows from $\gcd(n\pm4,n)=1$ that both $n+4$ and $n-4$ are squares, though that is impossible except when $n=5$.
 
 (iv) Assume that $\sqrt{n(n\pm4)} \in \Q(\zeta_n)$.
 Since $\gcd((n\pm4)/4,n)=1$, the same argument shows that $n/4\pm1$ are squares.
 This is a contradiction.
\end{proof}

\begin{proof}[Proof of Theorem~\ref{thm:PrimeDivMod}]
 It follows from Lemmas~\ref{lem:QuadResMod} and \ref{lem:Weintraub} that there are $\varepsilon \in \{1,-1\}$ and an odd prime $l$ such that $n(n+4\varepsilon)$ is a quadratic residue $\tmod l$ and $l \equiv m \mod n$.
 Therefore, by $\gcd(l,n)=1$, $n(n+4\varepsilon)$ is a quadratic residue $\tmod l$ if and only if $(2nx+n)^2 \equiv n^2+4\varepsilon n \mod l$ is solvable.
 Moreover, this congruence equation is equivalent to $nx(x+1) \equiv \varepsilon \mod l$.
\end{proof}

\begin{remark}
 The above proof (and the case $p=5$ in the proof of Theorem~\ref{thm:hc=1}) claims that for any $L(p,q')$ there exists an odd integer $q$ such that $L(p,q)$ is homeomorphic to $L(p,q')$ and $p(p+4)$ or $p(p-4)$ is a quadratic residue $\tmod q$, which is Sakasai's sufficient condition in Section~\ref{sec:Intro}.
 
 Even if $n=5$, Theorem~\ref{thm:PrimeDivMod} holds for $m \equiv 1,3,4 \mod 5$.
 Indeed, one can choose $(\varepsilon,l) = (-1,11),(1,3),(1,19)$ respectively.
 However, in the case $m \equiv 2 \mod 5$, Theorem~\ref{thm:PrimeDivMod} fails since neither $5(5+4)$ nor $5(5-4)$ is a quadratic residue $\tmod m$ by the quadratic reciprocity law.
\end{remark}

\section{Homologically fibered knots and the Alexander polynomial}\label{sec:HFK}

The aims of this section are to complete the proof of Corollary~\ref{cor:List_of_hc} and to characterize homologically fibered knots in a rational homology 3-sphere in terms of the Alexander polynomial.
These are achieved by Proposition~\ref{prop:SeifertMatrix} below.

\begin{figure}[h]
 \centering
 \includegraphics[height=8em]{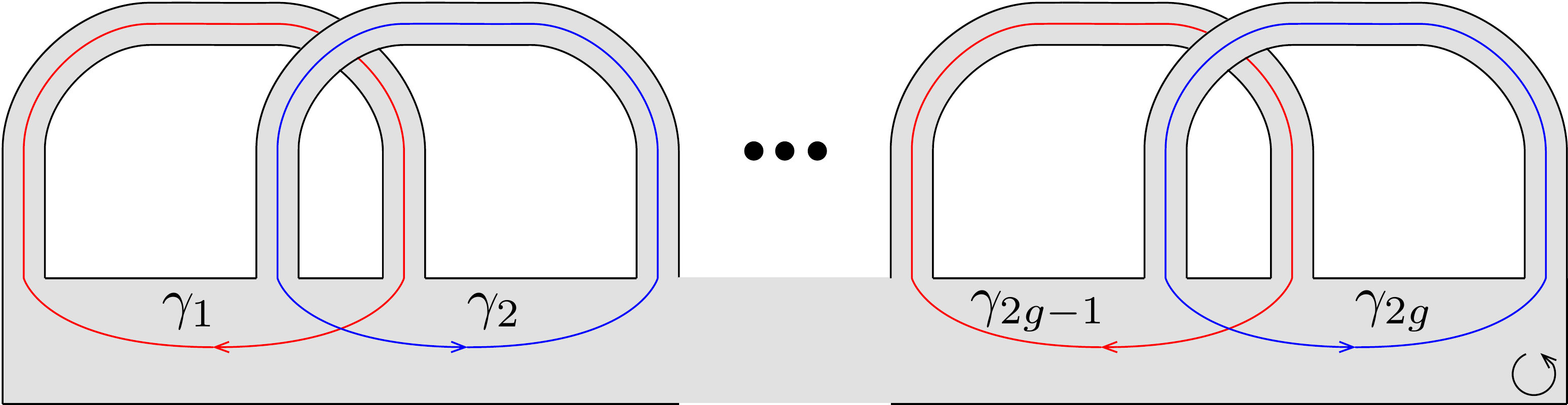}
 \caption{Oriented curves $\gamma_1,\dots,\gamma_{2g}$ on the surface $\Sigma_{g,1}$.}
 \label{fig:gammaOnSigma}
\end{figure}

Let $Y$ be a rational homology 3-sphere, $\iota \colon \Sigma_{g,1} \xrightarrow{\cong} \Sigma \subset Y$ an embedding.
Let $S=(s_{ij})_{i,j} \in M_{2g}(\Q)$ be the Seifert matrix of $\Sigma$ with respect to $\{\iota(\gamma_1),\dots, \iota(\gamma_{2g})\}$ (see Figure~\ref{fig:gammaOnSigma}), that is, $s_{ij} \in \Q$ is the linking number $\lk_Y(\iota(\gamma_i),\gamma_j^+)$ of two oriented curves $\iota(\gamma_i)$ and $\gamma_j^{+} := \iota^{+}(\gamma_j)$ in $Y$ (see \cite[Section~1.2]{Les96} for example).
Here, $S$ can be regarded as the matrix of the linear map $(\iota^+)_\ast\otimes\id_\Q \colon H_1(\Sigma_{g,1}; \Q) \to H_1(Y \setminus \Sigma; \Q)$ with respect to $\{[\gamma_1], \dots, [\gamma_{2g}]\}$ and $\{[\gamma_1'], \dots, [\gamma_{2g}']\}$, where $\gamma_i' \subset Y\setminus\Sigma$ is a meridian of $\iota(\gamma_i)$.
Indeed, by the definition of the linking number $\lk_Y$, we have
$$[\gamma_j^+] = \sum_{i=1}^{2g} \lk_Y(\iota(\gamma_i),\gamma_j^+) [\gamma_i'] \in H_1(Y \setminus \iota(\gamma_1\cup\dots\cup\gamma_{2g}); \Q) \xleftarrow{\cong} H_1(Y \setminus \Sigma; \Q).$$
Similarly, $S^T$ is regarded as the matrix of $(\iota^-)_\ast\otimes\id_\Q$.
The definition of $\lk_Y$ also implies that $S-S^T = \bigoplus^g
\begin{pmatrix}
 0 & -1 \\
 1 & 0
\end{pmatrix}
=:J$.

\begin{proposition}\label{prop:SeifertMatrix}
 $(Y\setminus\Int(\Sigma\times[-1,1]),\iota^-,\iota^+)$ is a homology cobordism if and only if $|H_1(Y)| |{\det S}| = 1$. 
\end{proposition}

\begin{proof}
 Let $TH_1(Y \setminus \Sigma)$ denote the torsion subgroup of $H_1(Y \setminus \Sigma)$, and put $F:=H_1(Y \setminus \Sigma)/T \cong \Z^{2g}$.
 Let $\B$ be a basis of $F$, $A_\pm \in GL(2g,\Z)$ the matrices of $\pr \circ (\iota^\pm)_\ast \colon H_1(\Sigma_{g,1}) \to F$ with respect to $\{[\gamma_1], \dots, [\gamma_{2g}]\}$ and $\B$.
 Since $\B$ can be regarded as a basis of $F \otimes \Q = H_1(Y \setminus \Sigma; \Q)$, we have the matrix $Q \in GL(2g,\Q)$ changing the basis $\B$ to $\{[\gamma_1'], \dots, [\gamma_{2g}']\}$.
 Then one has $Q^{-1}A_{+} = S$ and $Q^{-1}A_{-} = S^T = S-J$, and hence $A_{+}-A_{-} = QJ \in M_{2g}(\Z) \cap GL(2g,\Q)$.
 
 Let $\widetilde{\Sigma}:=\Sigma\times[-1,1]$.
 The Mayer-Vietoris sequence for $Y = (Y \setminus \Int\widetilde{\Sigma}) \cup \widetilde{\Sigma}$ gives the short exact sequence
 $$0 \to H_1(\partial\widetilde{\Sigma}) \xrightarrow{\phi} H_1(Y \setminus \Int\widetilde{\Sigma}) \oplus H_1(\Sigma) \to H_1(Y) \to 0.$$
 Identifying $H_1(\partial\widetilde{\Sigma})$ with $H_1(\Sigma\times\{1\}) \oplus H_1(\Sigma\times\{-1\})$, we have the commutative diagram
 \[\xymatrix{
 0 \ar[r] & \Z^{2g}\oplus\Z^{2g} \ar[r]^-{\phi}\ar@{=}[d] & (\Z^{2g}\oplus T)\oplus\Z^{2g} \ar[r]\ar@{->>}[d]^-{\pr} & H_1(Y) \ar[r]\ar@{->>}[d]^-{\overline{\pr}} & 0 \ (\text{exact}) \\
 0 \ar[r] & \Z^{2g}\oplus\Z^{2g} \ar[r]^-{\overline{\phi}} & \Z^{2g}\oplus\Z^{2g} \ar[r] & \Coker\overline{\phi} \ar[r] & 0 \ (\text{exact}), }
 \]
 where $\overline{\phi}$ and $\overline{\pr}$ are, respectively, the induced maps by $\phi$ and $\pr$ in the diagram.
 Here, the matrix of $\overline{\phi}$ is
 $\begin{pmatrix}
 A_+ & A_- \\
 I_{2g} & I_{2g}
\end{pmatrix}
 $, which is transformed to
 $\begin{pmatrix}
 QJ & A_- \\
 O & I_{2g}
\end{pmatrix}
 $ by elementary column operations.
 
 Suppose $|H_1(Y)| |{\det S}| = 1$.
 The exactness of the lower row implies that
 $$|{\Coker\overline{\phi}}| = |{\det Q}| = |{\det(A_+S^{-1})}| = |{\det A_{+}}| |H_1(Y)|.$$
 It follows from the surjectivity of $\overline{\pr}$ that $|{\det A_{+}}| = 1$, and hence
 $$|{\det A_{-}}| = |{\det(QS^T)}| = |{\det(QS)}| = 1.$$
 Finally, the five lemma shows that $TH_1(Y \setminus \Sigma) = 0$, and thus the maps $(\iota^\pm)_\ast$ are isomorphisms.
 
 Conversely, if $Y\setminus\Int(\Sigma\times[-1,1])$ is a homology cobordism, then we have $TH_1(Y \setminus \Sigma) = 0$ and $|{\det A_\pm}|=1$.
 It follows that
 $$|H_1(Y)||{\det S}| = |{\det QJ}||{\det Q^{-1}A_{+}}| = 1.$$
 This completes the proof.
\end{proof}

\begin{proof}[Alternative proof of Lemma~\textup{\ref{lem:Sabcuv}}]
 We first compute the Seifert matrix $S=(s_{ij})$ of $\Sigma=\Sigma_{a,b,c,u,v}$ with respect to $\{\alpha,\beta\}$.
 Let $\mu''$ be a parallel copy of $\mu'$ drawn in Figure~\ref{fig:Sabcuv} with $\lk_{S^3}(\mu',\mu'')=0$.
 We see that $\lk_{L(p,q)}(\mu',\mu'')=q/p$ by the definition of Dehn surgery and the linking number, and thus
 $$s_{12} = \lk_{L(p,q)}(\alpha,\beta^{+}) = c-uv\lk_{L(p,q)}(\mu',\mu'') = c-\frac{q}{p}uv.$$
 One can compute the other entries $s_{ij}$ similarly and conclude that
\begin{align}\label{eq:SeifertMatrix}
 S =
 \begin{pmatrix}
 a+\frac{q}{p}u^2 & c-\frac{q}{p}uv \\
 c+1-\frac{q}{p}uv & b+\frac{q}{p}v^2 \\
 \end{pmatrix}.
\end{align}
 Since
 \[p\det S = -p(c^2+c-ab)+q(bu^2+(2c+1)uv+av^2),\]
 if the complement is a homology cobordism, then Proposition~\ref{prop:SeifertMatrix} shows that there are $\varepsilon$ and $k$ as in Lemma~\ref{lem:Sabcuv}.
 Conversely, the existence of $\varepsilon$ and $k$ implies that $p|{\det S}| = 1$.
\end{proof}

The following terminology was introduced by Goda and Sakasai \cite[Definition~3.1]{GoSa13} in the case $X=S^3$ (see also \cite[Definition~7.1]{Sak17}).

\begin{definition}\label{def:HFK}
 An oriented knot $K$ in a connected oriented closed 3-manifold $X$ is called a \emph{homologically fibered knot of genus $g$} if there is a Seifert surface of $K$ such that $X\setminus\Int(\Sigma\times[-1,1])$ is a homology cobordism over $\Sigma_{g,1}$.
\end{definition}

By definition, $\hc(X)=g$ if and only if $X$ contains a homologically fibered knot of genus $g$, but does not contain one of genus $g-1$.

\begin{remark}
 If $X$ is a rational homology 3-sphere, then Corollary~\ref{cor:AlexanderPolynomial} shows that $g$ in Definition~\ref{def:HFK} must be equal to the \emph{knot genus} $g(K)$ of $K$.
\end{remark}

We next see that homologically fibered knots are characterized by the Alexander polynomial.
Let $K$ be an oriented knot in a rational homology 3-sphere $Y$.
We define the \emph{Alexander polynomial} $\Delta_K(t)$ of $K$ by
\[\Delta_K(t) := |H_1(Y)|\det(t^{1/2}S-t^{-1/2}S^T) \in \Z[t,t^{-1}],\]
where $S$ is a Seifert matrix of a Seifert surface $\Sigma$ of $K$ (see, for example, \cite[Proposition~2.3.13]{Les96}).
By definition, $\Delta_K(t)$ should be palindromic and satisfy $\Delta_K(1) = |H_1(Y)|$, and its breadth is less than or equal to $2g(K)$.

\begin{example}
 Suppose that $L(p,q) \setminus \Int(\Sigma_{a,b,c,u,v}\times[-1,1])$ is a homology cobordism (see Figure~\ref{fig:Sabcuv}).
 Then there exists $\varepsilon \in \{1,-1\}$ as in Lemma~\ref{lem:Sabcuv}, and using the Seifert matrix \eqref{eq:SeifertMatrix}, one can compute the Alexander polynomial of $K=\partial\Sigma_{a,b,c,u,v}$ in $L(p,q)$:
 \[\Delta_K(t) = p-\varepsilon(t-2+t^{-1}).\]
\end{example}

The next result is a corollary of Proposition~\ref{prop:SeifertMatrix}, which is well-known in the case $Y=S^3$ (see \cite[Proposition~3.2]{GoSa13}).

\begin{corollary}\label{cor:AlexanderPolynomial}
 An oriented knot $K$ in $Y$ is homologically fibered if and only if $\Delta_K(t)$ is monic \textup{(}up to sign\textup{)} and its breadth equals $2g(K)$.
\end{corollary}

\begin{proof}
 In general, for $S \in GL(2g,\Q)$, the highest degree term of $\det(t^{1/2}S-t^{-1/2}S^T)$ is equal to $(\det S)t^g$.
 Therefore, Proposition~\ref{prop:SeifertMatrix} proves the corollary.
 \end{proof}

Finally, we complete the rest of the proof of Corollary~\ref{cor:List_of_hc} by using Theorem~\ref{thm:hc=1} and Proposition~\ref{prop:SeifertMatrix}.

\begin{figure}[h]
 \centering
 \includegraphics[height=10em]{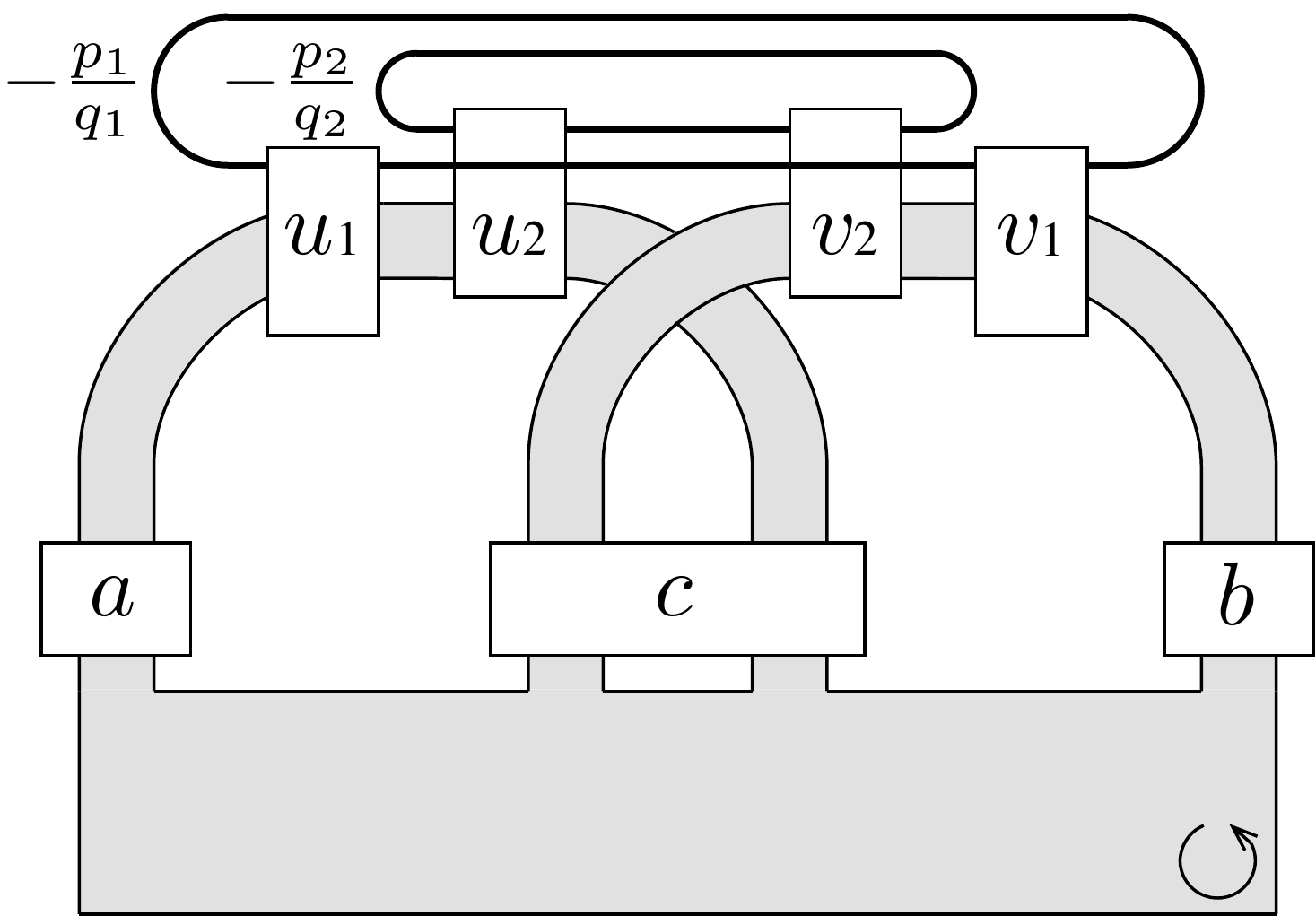}
 \caption{The surface $\Sigma_{a,b,c,u_1,u_2,v_1,v_2} \subset \nobreak L(p_1,q_1)\sharp L(p_2,q_2)$, where the indices of $\Sigma$ are integers like Figure~\ref{fig:Sabcuv}.} 
 \label{fig:Sabcuuvv}
\end{figure}

\begin{proof}[Proofs of Corollary~\textup{\ref{cor:List_of_hc}} \textup{(5) and (6)}.]
 (5) For $X=L(p_1,q_1) \sharp L(p_2,q_2)$ with $p_1 \mid p_2$, we know that $\hc(X)$ equals 1 or 2 by Theorem~\ref{thm:hc=1}.
 Suppose that $\hc(X)=1$, namely there is $\Sigma_{1,1}$ in $X$ whose complement is a homology cobordism.
 We find a surface of the form $\Sigma_{a,b,c,u_1,u_2,v_1,v_2}$ (see Figure~\ref{fig:Sabcuuvv}) whose Seifert matrix $S$ is same as a Seifert matrix of $\Sigma_{1,1}$.
 Then, by an argument similar to the alternative proof of Lemma~\ref{lem:Sabcuv}, one concludes that
 \[S = 
 \begin{pmatrix}
 a+\frac{q_1}{p_1}u_1^2+\frac{q_2}{p_2}u_2^2 & c-\frac{q_1}{p_1}u_1v_1-\frac{q_2}{p_2}u_2v_2 \\
 c+1-\frac{q_1}{p_1}u_1v_1-\frac{q_2}{p_2}u_2v_2 & b+\frac{q_1}{p_1}v_1^2+\frac{q_2}{p_2}v_2^2 \\
 \end{pmatrix}
 \in \frac{1}{p_2}M_2(\Z).\]
 It follows from Proposition~\ref{prop:SeifertMatrix} that
 \[1 = |H_1(X)||{\det S}| \equiv \varepsilon q_1q_2(u_1u_2-v_1v_2)^2 \mod p_1\]
 for some $\varepsilon \in \{1,-1\}$.
 Thus, $q_1q_2$ or $-q_1q_2$ is a quadratic residue $\tmod p_1$.
 
 (6) Theorem~\ref{thm:Sakasai} allows us to assume that $X=(S^1\times S^2) \sharp L(p,q)$.
 By Theorem~\ref{thm:hc=1}, it suffices to prove that $\hc(X)=1$ if and only if $q$ or $-q$ is a quadratic residue $\tmod p$.
 Suppose that $\hc(X)=1$.
 There is $A \in SL(2,\Z)$ such that the closure $C$ of the homology cobordism $(\Sigma_{1,1}\times[-1,1],\incl,\widetilde{A})$ is Borromean surgery equivalent to $X$, where $\widetilde{A}$ is a homeomorphism of $\Sigma_{1,1}$ inducing $A$ on $H_1(\Sigma_{1,1})$ (see \cite[Section~6.3]{Sak17}).
 Here we have $\det(A-I_2)=0$ since the Mayer-Vietoris sequence for $C = \Sigma_{1,1}\times[-1,0] \cup \Sigma_{1,1}\times[0,1]$ shows that $\Coker(A-I_2) \cong H_1(C)$.
 Therefore, by the proof of \cite[Proposition~2]{TaOc82}, $A$ is conjugate to 
$\begin{pmatrix}
 1 & 0 \\
 u & 1
\end{pmatrix}$
 in $SL(2,\Z)$ for some $u \in \Z$.
 Now, $u$ must be $\varepsilon p$ for some $\varepsilon \in \{1,-1\}$, and one can choose $\widetilde{A}$ so that $C \cong (S^1\times S^2)\sharp L(p,\varepsilon)$.
 Thus the torsion linking form $(q/p)$ is isomorphic to $(\varepsilon/p)$.
 
 In particular, the above argument implies that $(S^1\times S^2)\sharp L(p,\varepsilon)$ is the closure of a homology cobordism over $\Sigma_{1,1}$ for each $\varepsilon=\pm1$, which proves the converse.
\end{proof}

\end{document}